\documentclass{amsart}
\usepackage{amsmath}
\usepackage{amssymb}
\usepackage{amsthm}
\usepackage{verbatim}
\usepackage{graphicx}
\usepackage{graphics}
\usepackage{color}
\usepackage{enumerate}
\usepackage{mathrsfs}
\usepackage{booktabs}
\usepackage[toc,page]{appendix}
\usepackage{enumerate}
\usepackage{mathrsfs}
\usepackage{fancyhdr}
\usepackage{latexsym, amsfonts, amssymb, amsmath,  amsthm}
\usepackage{amsopn, amstext, amscd,pifont}
\usepackage{amssymb,bm, cite}
\usepackage[all]{xy}
\usepackage{color}
\usepackage{graphicx,hyperref}
\usepackage{float}
\usepackage{listings}
\usepackage{setspace}
\usepackage[T1]{fontenc}
\usepackage{url}
 \usepackage{paralist}

\usepackage{graphicx,pgf}
\usepackage{float}

\usepackage{fancyhdr}

     \newcommand{\sM}{\mathcal M}
     
     \newcommand{\sO}{\mathcal O}

     \newcommand{\sR}{\mathcal R}

\def\XXint#1#2#3{{\setbox0=\hbox{$#1{#2#3}{\int}$ }
\vcenter{\hbox{$#2#3$ }}\kern-.6\wd0}}

\newcommand{\C}{\mathbb{C}}

\newcommand{\F}{\mathbb{F}}

\newtheoremstyle{ser}
{8pt}
{8pt}
{\it}
{}
{\sf}
{:}
{6mm}
{}

\newtheoremstyle{serr}
{8pt}
{8pt}
{\normalfont}
{}
{\sf}
{.}
{6mm}
{}

\theoremstyle{r}

\newtheorem{theorem}{Theorem}[section]
\newtheorem{definition}[theorem]{Definition}
\newtheorem{example}[theorem]{Example}

\newtheorem{proposition}[theorem]{Proposition}
\newtheorem{lemma}[theorem]{Lemma}

\newtheorem{corollary}[theorem]{Corollary}
\newtheorem{remark}[theorem]{Remark}

\graphicspath{{images/}}

\begin{document}

\title{Rescaling Limits in Non-Archimedean Dynamics}

\author{Hongming Nie}
\address{Indiana University, 831 East Third Street, Rawles Hall, Bloomington, Indiana 47405, USA}
\email{nieh@indiana.edu}
\maketitle

\begin{abstract}
Suppose $\{f_t\}$ is an analytic one-parameter family of rational maps defined over a non-Archimedean field $K$. We prove a finiteness theorem for the set of rescalings for $\{f_t\}$. This complements results of J. Kiwi.
\end{abstract}
\section{Introduction}
Let $K$ be an algebraically closed field. For $d\ge 1$, let $\mathrm{Rat}_d(K)$ be the space of degree $d$ rational maps over $K$, thought of as dynamical systems on $\mathbb{P}^1(K)$. The group $\mathrm{PGL}_2(K)$ acts on $\mathrm{Rat}_d(K)$ by conjugation. The moduli space of degree $d$ rational maps on $\mathbb{P}^1(K)$ is the quotient space $\sM_d(K):=\mathrm{Rat}_d(K)/\mathrm{PGL}_2(K)$. Milnor \cite{Mi1} considered the moduli space $\sM_2(\C)$ of quadratic complex rational maps and gave a dynamically natural compactification of $\sM_2(\C)$. Then using geometric invariant theory, Silverman \cite{Si1,Si2} compactified the moduli space $\sM_d(K)$ in general. DeMarco \cite{De2} also considered different compactifications of the moduli space $\sM_2(\C)$. To study the dynamics of complex rational maps approaching the boundary of $\mathrm{Rat}_d(\C)$ (or $\sM_d(\C)$), Kiwi \cite{Ki} considered rescaling limits for a holomorphic family $\{f_t\}$ (resp. a sequence $\{f_n\}$) in $\mathrm{Rat}_d(\C)$. These arise as limits $M_t^{-1}\circ f_t^q\circ M_t\to g$ (resp. $M_n^{-1}\circ f_n^q\circ M_n\to g$) of rescaled iterates where the convergence is locally uniform outside some finite subset of $\mathbb{P}^1(\C)$. By regarding a holomorphic family as a rational map with coefficients in the field of formal Puiseux series, and by studying its induced action on the corresponding Berkovich space, Kiwi proved for any given holomorphic one-parameter family of degree $d\ge 2$ rational maps, there are at most $2d-2$ dynamically independent rescalings such that the corresponding rescaling limits are not postcritically finite \cite[Theorem 1, Theorem 2]{Ki}. Later, Arfeux \cite{Ar1} proved the same results using the Deligne-Mumford compactifications of the moduli spaces of the stable punctured spheres.\par
An algebraically closed complete valued field is isomorphic to either the field of complex numbers $\C$ or a non-Archimedean field \cite{DR}. Our main result translates Kiwi's finiteness result to non-Archimedean algebraically closed fields, subject to a natural tameness hypothesis. We now set up the statement.\par
Throughout this paper, $K$ will denote an algebraically closed field which is complete with respect to a nontrivial non-Archimedean absolute value $|\cdot|_K$. Let $\phi(z)\in K(z)$ be a rational map. We can write $\phi=\phi_1\circ\phi_2$, where $\phi_1$ is a separable rational map and $\phi_2(z)=z^{p^j}$ for some $j\ge 0$ if the field $K$ has positive characteristic $p>0$ or $\phi_2(z)=z$ if the field $K$ has characteristic zero. The rational map $\phi_1$ is called the \textit{separable part} of $\phi$. The degree of $\phi_1$ is called the \textit{nontrivial degree} of $\phi$ and the preimages of critical points of $\phi_1$ under $\phi_2$ are called the \textit{nontrivial critical points} of $\phi$. We say rational map $\phi\in K(z)$ is \textit{postcritically finite at nontrivial critical points} if each nontrivial critical point of $\phi$ has a finite forward orbit; equivalently, each critical value of $\phi_1$ has a finite forward orbit under $\phi$.\par
As in Kiwi's study of degenerating rational maps over $\C$, we now study families $\phi_t$ approaching boundary of $\mathrm{Rat}_d(K)$ (or $\sM_d(K)$). Since the field $K$ is not locally compact with respective to the absolute value $|\cdot|_K$, the definition of rescaling limits in \cite{Ki} needs to be slightly modified. The non-Archimedean property turns out to make pointwise convergence suitable; see Definition \ref{rescaling} and Proposition \ref{ptwise-coeff}. For instance, let $f_t(z)=(z^3+t)/z\in K(z)$, then, as $t\to 0$, $f_t(z)$ converges to $z^2$ pointwise on $\mathbb{P}^1(K)\setminus\{0\}$, but $f_t(0)=\infty$ for all $t\not=0$.  Since the field of formal Puiseux series over $K$ is not algebraically closed if $char\ K=p>0$, we work on the field $\mathbb{L}$ of Hahn series over $K$. For Puiseux series and Hahn series, we refer \cite{ Hahn, Ke,Ki1, Ki2}.  For an analytic family $\{f_t\}\subset \mathrm{Rat}_d(K)$, we can associate to $\{f_t\}$ a rational map $\mathbf{f}:\mathbb{P}^1(\mathbb{L})\to\mathbb{P}^1(\mathbb{L})$. The space $\mathbb{P}^1(\mathbb{L})$ is naturally a subset of the corresponding Berkovich space $\mathbb{P}^1_{\mathrm{Ber}}(\mathbb{L})$; see \cite{BR,Jo} for details. The rational map $\mathbf{f}$ induces a map on $\mathbb{P}_{\mathrm{Ber}}^1(\mathbb{L})$ extending its action on $\mathbb{P}^1(\mathbb{L})$, so we also use notation $\mathbf{f}$ for the induced map. Let $\sR_\mathbf{f}$ be the Berkovich ramification locus, that is the set of points in $\mathbb{P}_{\mathrm{Ber}}^1(\mathbb{L})$ such that the local degrees of $\mathbf{f}$ at these points are at least $2$, i.e $\sR_\mathbf{f}=\{\xi\in\mathbb{P}_{\mathrm{Ber}}^1(\mathbb{L}):\text{deg}_\xi\textbf{f}\ge 2\}$. It is a closed subset of $\mathbb{P}_{\mathrm{Ber}}^1(\mathbb{L})$ with no isolated points and has at most $\deg\mathbf{f}-1$  connected components \cite[Theorem A]{Fa1}, each of which has tree structure. Following Trucco \cite{Tr}, we say the rational map $\mathbf{f}:\mathbb{P}_{\mathrm{Ber}}^1(\mathbb{L})\to\mathbb{P}_{\mathrm{Ber}}^1(\mathbb{L})$ is \textit{tame} if $\sR_{\mathbf{f}}$ has only finitely many points with valence at least $3$ in $\sR_{\mathbf{f}}$. \par
We will prove
\begin{theorem}\label{main-thm}
Let $\{f_t(z)\}\subset K(z)$ be an analytic family of rational maps of nontrivial degree $d\ge 2$ and let $\mathbf{f_1}$ be the separable part of the associated rational map $\mathbf{f}$ of $\{f_t\}$. Assume $\mathbf{f_1}$ is tame. Then there are at most $2d-2$ pairwise dynamically independent rescalings for $\{f_t\}$ such that the corresponding rescaling limits are not postcritically finite at nontrivial critical points.
\end{theorem}
In section \ref{ex}, we give some examples of analytic families with rescaling limits that are not postcritically finite at nontrivial critical points.\par
The tameness hypothesis of the separable part $\mathbf{f_1}$ is needed in order to prove a non-Archimedean Rolle's theorem in positive characteristic, see Lemma \ref{rolle-positive}. If $K$ has characteristic zero or positive characteristic $p>\text{deg}\ \mathbf{f_1}$, then the rational map $\mathbf{f_1}:\mathbb{P}_{\mathrm{Ber}}^1(\mathbb{L})\to\mathbb{P}_{\mathrm{Ber}}^1(\mathbb{L})$ is tame \cite[Corollary 6.6]{Fa1}. Thus
\begin{corollary}
Suppose the field $K$ has characteristic zero. Let $\{f_t(z)\}\subset K(z)$ be an analytic family of degree $d\ge 2$ rational maps. Then there are at most $2d-2$ pairwise dynamically independent rescalings for $\{f_t\}$ such that the corresponding rescaling limits are not postcritically finite.
\end{corollary}
\subsection*{Outline}
In section \ref{pr}, we recall the relevant backgrounds of Berkovich space and define the rescaling limits for an analytic family of rational maps over $K$. The goal of section \ref{re} is to discuss the reduction map and show the relations between reductions and rescaling limits. Section \ref{bd} is devoted to restating Kiwi's results which are still true for the case when $K$ has characteristic zero. Finally, we prove Theorem \ref{main-thm} in section \ref{pc} and give examples to illustrate it in section \ref{ex}.

\section{Preliminaries}\label{pr}
\subsection{Non-Archimedean fields}
For the field $K$, let $|K^\times|_K\subset(0,\infty)$ be the set of absolute values attained by nonzero elements of $K$, which is called the value group of $K$. Then $|K^\times|_K$ is dense in $(0,\infty)$ since $K$ is algebraically closed, and hence $K$ can not be locally compact. Let $\mathcal{O}_K=\{z\in K:|z|_K\le 1\}$ be the ring of integers of $K$ and let $\mathcal{M}_K=\{z\in K:|z|_K<1\}$ be the unique maximal ideal of $\mathcal{O}_K$. Let $k=\mathcal{O}_K/\mathcal{M}_K$ be the residue field. Note if $char\ K=p>0$ then $char\ k=p$, but if $char\ K=0$, then $k$ could have any characteristic. For instance, for a prime number $p\ge 2$, if $K$ is the completion of the algebraic closure of the formal power series field $\F_p[[t]]$ with its natural absolute value, then $char\ K=char\ k=p$; if $K$ is the complex $p-$adic field $\C_p$, then $k=\overline{\F}_p$, the algebraic closure of $\F_p$, and $char\ K=0$ but $char\ k=p$.\par
Given $a\in K$ and $r>0$, define 
$$D(a,r):=\{z\in K: |z-a|_K<r\}\ \ \text{and}\ \ \overline{D}(a,r):=\{z\in K: |z-a|_K\le r\}.$$
If $r\in|K^\times|_K$, we say that $D(a,r)$ is an open rational disk in $K$ and $\overline{D}(a,r)$ is a closed rational disk in $K$. If $r\not\in|K^\times|_K$, we call $D(a,r)=\overline{D}(a,r)$ an irrational disk. Let $U(a,r)\subset K$ be a disk centered at $a\in K$ with radius $r>0$, that is, $U$ has the form $D(a,r)$ or $\overline{D}(a,r)$. Then if $b\in U(a,r)$, we have $U(a,r)=U(b,r)$. Moreover, the radius $r$ is the same as the diameter of $U(a,r)$, that is $r=\sup\{|z-w|_K:z,w\in U(a,r)\}$. Furthermore, if two disks have a nonempty intersection, then one must contain the other. Finally, we should mention here every disk in $K$ is both open and closed under the topology of $K$. \par
Let $\mathbb{L}:=K[[t^\mathbb{Q}]]$ be the field of Hahn series over $K$. It consists of all formal sums of the form $\sum_{n\ge 0}a_nt^{q_n}$, where $\{q_n\}$ is an increasing sequence of rational numbers and $a_n\in K$. Since $\mathbb{Q}$ is divisible under addition, the field $\mathbb{L}$ is algebraically closed. It can be equipped with a non-Archimedean absolute value $|\cdot|_{\mathbb{L}}$ by fixing a number $\epsilon\in(0,1)$ and defining $|\sum_{n\ge 0}a_nt^{q_n}|_{\mathbb{L}}=\epsilon^{n_0}$, where $n_0$ is the smallest positive integer such that $a_n\not=0$. With respect to $|\cdot|_{\mathbb{L}}$, the field $\mathbb{L}$ is complete. Then the ring of integer of the field $\mathbb{L}$ is 
$$\mathcal{O}_\mathbb{L}=\{z\in\mathbb{L}:|z|_\mathbb{L}\le 1\}=\left\{\sum\limits_{n\ge 0}a_nt^{q_n}:q_n\ge 0\right\}$$
and the unique maximal ideal $\sM_\mathbb{L}$ of $\sO_\mathbb{L}$ consists of series with zero constant term, i.e.
$$\mathcal{M}_\mathbb{L}=\{z\in\mathbb{L}:|z|_\mathbb{L}<1\}=\left\{\sum\limits_{n\ge 0}a_nt^{q_n}:q_n> 0\right\}.$$
The residue field $\mathcal{O}_\mathbb{L}/\mathcal{M}_\mathbb{L}$ is canonically isomorphic to $K$.

\subsection{The Berkovich projective line}
In this subsection, we summarize some fundamental properties of the Berkovich projective line, for details we refer \cite{BR,Ben,Be}.\par
The Berkovich affine line $\mathbb{A}_{\mathrm{Ber}}^1(K)$ is the set of all multiplicative seminorms on the ring $K[z]$ of polynomials over $K$, whose restriction to the field $K\subset K[z]$ is equal to the given absolute value $|\cdot|_K$. For $a\in K$ and $r\ge 0$, let $\xi_{a,r}$ be the seminorm defined by $|f|_{\xi_{a,r}}=\sup_{z\in\overline{D}(a,r)}|f(z)|_K$. Then there are $4$ types of points in $\mathbb{A}_{\mathrm{Ber}}^1(K)$:\par
1. Type I. $\xi_{a,0}$ for some $a\in K$.\par
2. Type II. $\xi_{a,r}$ for some $a\in K$ and $r\in|K^\times|$. \par
3. Type III. $\xi_{a,r}$ for some $a\in K$ and $r\notin|K^\times|$. \par 
4. Type IV. A limit of seminorms $\{\xi_{a_i,r_i}\}_{i\ge 0}$, where the corresponding sequence of closed disks $\{\overline{D}_{a_i,r_i}\}_{i\ge 0}$ satisfies $\overline{D}_{a_{i+1},r_{i+1}}\subset\overline{D}_{a_i,r_i}$ and $\cap\overline{D}_{a_i,r_i}=\emptyset$.\par 
We can identify $K$ with the type I points in $\mathbb{A}_{\mathrm{Ber}}^1(K)$ via $a\to\xi_{a,0}$. The point $\xi_{0,1}\in\mathbb{A}_{\mathrm{Ber}}^1(K)$ is called the Gauss point and denoted by $\xi_G$. We put the weak topology on $\mathbb{A}_{\mathrm{Ber}}^1(K)$, which makes the map $\mathbb{A}_{\mathrm{Ber}}^1(K)\to[0,+\infty)$ sending $\xi$ to $|f|_\xi$ continuous for each $f\in K[z]$. Then $\mathbb{A}_{\mathrm{Ber}}^1(K)$ is locally compact, Hausdorff and uniquely path-connected.\par
The Berkovich projective line $\mathbb{P}_{\mathrm{Ber}}^1(K)$ is obtained by gluing two copies of $\mathbb{A}_{\mathrm{Ber}}^1(K)$ along $\mathbb{A}_{\mathrm{Ber}}^1(K)\setminus\{0\}$ via the map $\xi\to 1/\xi$. Then we can associate the Gelfand topology on $\mathbb{P}_{\mathrm{Ber}}^1(K)$. The Berkovich projective line $\mathbb{P}_{\mathrm{Ber}}^1(K)$ is a compact, Hausdorff, uniquely path-connected topological space and contains $\mathbb{P}^1(K)$ as a dense subset.\par
The space $\mathbb{P}_{\mathrm{Ber}}^1(K)$ has tree structure. For a point $\xi\in\mathbb{P}_{\mathrm{Ber}}^1(K)$, we can define an equivalence relation on $\mathbb{P}_{\mathrm{Ber}}^1(K)\setminus\{\xi\}$, that is, $\xi'$ is equivalent to $\xi''$ if $\xi'$ and $\xi''$ are in the same connected component of $\mathbb{P}_{\mathrm{Ber}}^1(K)\setminus\{\xi\}$. Such an equivalence class $\vec{v}$ is called a direction at $\xi$. We say that the set $T_{\xi}\mathbb{P}_{\mathrm{Ber}}^1(K)$ formed by all directions at $\xi$ is the tangent space at $\xi$. For $\vec{v}\in T_{\xi}\mathbb{P}_{\mathrm{Ber}}^1(K)$, denote by $\mathbf{B}_{\xi}(\vec{v})^-$ the component of $\mathbb{P}_{\mathrm{Ber}}^1(K)\setminus\{\xi\}$ corresponding to the direction $\vec{v}$. If $\xi\in\mathbb{P}_{\mathrm{Ber}}^1(K)$ is a type I or IV point, $T_{\xi}\mathbb{P}_{\mathrm{Ber}}^1(K)$ consists of a single direction. If $\xi\in\mathbb{P}_{\mathrm{Ber}}^1(K)$ is a type II point, the directions in $T_{\xi}\mathbb{P}_{\mathrm{Ber}}^1(K)$ are in one-to-one correspondence with the elements in $\mathbb{P}^1(k)$. If $\xi\in\mathbb{P}_{\mathrm{Ber}}^1(K)$ is a type III point, $T_{\xi}\mathbb{P}_{\mathrm{Ber}}^1(K)$ consists of two directions. Note the Gauss point $\xi_G$ is a type II point. We can identify $T_{\xi_G}\mathbb{P}_{\mathrm{Ber}}^1(K)$ to $\mathbb{P}^1(k)$ by the correspondence $T_{\xi_G}\mathbb{P}_{\mathrm{Ber}}^1(K)\to\mathbb{P}^1(k)$ sending $\vec{v}_x$ to $x$, where $\vec{v}_x$ is the direction at $\xi_G$ such that $\mathbf{B}_{\xi_G}(\vec{v}_x)^-$ contains all the type I points whose images are $x$ under the canonical reduction map $\mathbb{P}^1(K)\to\mathbb{P}^1(k)$.
\subsection{Rational maps}
In this subsection, we consider rational maps over the field $K$ and define an analytic family of rational maps over $K$. For rational maps over a non-Archimedean field, we refer \cite{BR,Ben,Ben1}.\par
We first define analytic maps on a disk $U\subset K$.
\begin{definition}
Let $U\subset K$ be a disk and $z_0\in U$. We say a map $f:U\to K$ is \textit{analytic} if $f$ can be written as a power series 
$$f(z)=\sum_{n=0}^\infty c_n(z-z_0)^n\in K[[z-z_0]],$$
which converges for all $z\in U$. The smallest $n$ such that $c_n\not=0$ is called the \textit{order} of $f$ at $z_0$ and denoted $\mathrm{ord}_{z_0}(f)$.
\end{definition}
It is easy to check that analytic property is independent of the choice of $z_0\in U$. Moreover, if $U=\overline{D}(z_0,r)$ is a rational closed disk, then $\sum_{n=0}^\infty c_n(z-z_0)^n$ converges for each $z\in U$ if and only if $\lim_{n\to\infty} |c_n|_Kr^n=0$. For rational open or irrational disks, $\lim_{n\to\infty} |c_n|_Kr^n=0$ implies $\sum_{n=0}^\infty c_n(z-z_0)^n$ converges, but the converse is not true.\par
We denote by $\mathbb{P}^1(K):=K\cup\{\infty\}$ the projective line over $K$. We define the spherical metric on $\mathbb{P}^1(K)$ as follows: for points $z=[x:y]$ and $w=[u:v]$ in $\mathbb{P}^1(K)$,
$$\Delta(z,w):=\frac{|xv-yu|_K}{\max\{|x|_K,|y|_K\}\max\{|u|_K,|v|_K\}}.$$
Equivalently, 
$$\Delta(z,w):=
\begin{cases}
\frac{|z-w|_K}{\max\{1,|z|_K\}\max\{1,|w|_K\}}, &\text{if}\  z,w\in K,\\
\frac{1}{\max\{1,|z|_K\}}, &\text{if}\  z\in K,w=\infty.
\end{cases}$$\par
Recall a degree $d\ge 1$ rational map $f:\mathbb{P}^1(K)\to\mathbb{P}^1(K)$ is represented by a pair $f_1,f_2\in K[X,Y]$ of degree $d$ homogeneous polynomials with no common factors, that is, $f([X:Y])=[f_1(X,Y):f_2(X,Y)]$ for all $[X:Y]\in\mathbb{P}^1(K)$. Equivalently, the map $f$ can be considered as the quotient of two relatively prime polynomials, of which the greatest degree is $d$. Let $\mathrm{Rat}_d(K)$ denote the set of rational maps of degree $d$ over $K$. Then $\mathrm{Rat}_d(K)$ can be naturally identified with an open subset of $\mathbb{P}^{2d+1}(K)$ via the map $\mathrm{Rat}_d(K)\to\mathbb{P}^{2d+1}(K)$ sending $(a_dz^d+\cdots+a_0)/(b_dz^d+\cdots+b_0)$ to $[a_d:\cdots:a_0:b_d:\cdots:b_0]$.\par
Let $\phi(z)\in K(z)$ be a rational map. Suppose $z_0\in\mathbb{P}^1(K)$ and set $w_0=\phi(z_0)$. Pick $\psi_1,\psi_2\in\mathrm{PGL}_2(K)$ such that $\psi_1(0)=z_0$ and $\psi_2(w_0)=0$, and define $\Phi=\psi_2\circ\phi\circ\psi_1$. The multiplicity $m_\phi(z_0)$ of $\phi$ at $z_0$ is the order of $\Phi$ at $0$. The weight $w_\phi(z_0)$ of $\phi$ at $z_0$ is the order of $\Phi'$ at $0$. If $\Phi'(z)\equiv 0$, we set $w_\phi(z_0)=\infty$. This can happen: for example, if $char\ K=p$ and $\phi(z)=z^p$, then $\phi'(z)=0$ for each $z\in K$. A point $z_0\in\mathbb{P}^1(K)$ is called a critical point of $\phi$ if $w_\phi(z_0)>0$. Denote $\mathrm{Crit}(\phi)\subset\mathbb{P}^1(K)$ for the set of all critical points of $\phi$. If every point  $z\in\mathbb{P}^1(K)$ is a critical point of $\phi$, then we say $\phi$ is inseparable. Otherwise, $\phi$ is called separable. Recall that for every rational map $\phi\in K(z)$, we can write $\phi(z)=\phi_1\circ\phi_2(z)$, where $\phi_1$ is a separable rational map and $\phi_2(z)=z^{p^j}$ for some $j\ge 0$ if the field $K$ has positive characteristic $p>0$ or $\phi_2(z)=z$ if the field $K$ has characteristic zero. It is called the (in)separable decomposition of $\phi$. The rational map $\phi_1$ is called the separable part of $\phi$. We define the nontrivial degree $\text{deg}_0\phi:=\text{deg}\phi_1$ and the nontrivial critical set $\mathrm{Crit}_0(\phi):=\phi_2^{-1}(\mathrm{Crit}(\phi_1))$ of $\phi$.\par
\begin{definition}
Let $U\subset K$ be a disk containing $0$. A collection $\{f_t\}_{t\in U}\subset\mathbb{P}^{2d+1}(K)$ is a \textit{$1$-dimensional separable analytic family of degree $d\ge 1$ rational maps} if the map $F:U\to\mathbb{P}^{2d+1}(K)$ sending $t$ to $f_t$ is an analytic map such that $f_t\in\mathrm{Rat}_d(K)$ is separable for all $t\not=0$. If $\{f_t\}_{t\in U}$ is a $1$-dimensional separable analytic family of degree $1$ rational maps, we call it a moving frame.
\end{definition}
Let $U\subset K$ be a disk containing $0$. We say $\{f_t\}_{t\in U}\subset\mathbb{P}^{2d+1}(K)$ is a \textit{$1$-dimensional analytic family of nontrivial degree $d'\ge 1$ rational maps} if we can write $f_t=g_t\circ h$, where $\{g_t\}_{t\in U}\subset\mathbb{P}^{2d'+1}(K)$ is a $1$-dimensional separable analytic family of degree $d'\ge 1$ rational maps and $h(z)=z^{p^j}$ for some $j\ge 0$ if $char\ K=p>0$ or $h(z)=z$ if $char\ K=0$.
\begin{remark}
\begin{enumerate}
\item We are really interested in the germ defined by an analytic family, so considering a small disk $V\subset U$ containing $0$ if necessary, we can always assume $U=\overline{D}(0,r)$ is a rational closed disk.
\item For an analytic family $\{f_t\}$ on $U$, we can write 
$$f_t(z)=P_t(z)/Q_t(z)=(a_d(t)z^d+\cdots+a_0(t))/(b_d(t)z^d+\cdots+b_0(t))$$
and denote by $\ell$ the minimum among the orders of the $a_i(t)$ and $b_j(t)$, at the origin, $i,j=1,\cdots,d$.
Let 
$$C=\max_{0\le i,j\le d}\{1,|\lim\limits_{t\to 0}a_i(t)/t^\ell|_K,|\lim\limits_{t\to 0}b_j(t)/t^\ell|_K\}$$
Let $x\in K$ be an element such that $|x|_K=C$. For $t$ sufficiently small, by considering $g_t(z)=(P_t(z)/xt^\ell)/(Q_t(z)/xt^\ell)$ if necessary, we can assume $\{f_t\}\subset\mathcal{O}_K(z)$ and that $f_t$ has at least one coefficient with absolute value $1$. Therefore, throughout this paper, for a rational map $\phi(z)\in K(z)$, we assume $\phi(z)\in\mathcal{O}_K(z)$ and at least one coefficient has absolute value $1$. 
\end{enumerate}
\end{remark}
For an analytic family 
$$\left\{f_t(z)=\frac{a_d(t)z^d+\cdots+a_0(t)}{b_d(t)z^d+\cdots+b_0(t)}\right\}\subset K(z)$$ 
of degree $d$ rational maps, let $\mathbf{a_d},\cdots,\mathbf{a_0},\mathbf{b_d},\cdots,\mathbf{b_0}$ be the power series expressions of the coefficients $a_d(t),\cdots,a_0(t),b_d(t),\cdots,b_0(t)$, respectively. Then the degree $d$ rational map $\mathbf{f}:\mathbb{P}^1(\mathbb{L})\to\mathbb{P}^1(\mathbb{L})$ given by 
$$\mathbf{f}(z)=\frac{\mathbf{a_d}z^d+\cdots+\mathbf{a_0}}{\mathbf{b_d}z^d+\cdots+\mathbf{b_0}}$$
is called, following Kiwi, the rational map associated to $\{f_t\}$. The rational map $\mathbf{f}$ induces a map from $\mathbb{P}_{\mathrm{Ber}}^1(\mathbb{L})$ to itself. We use the same notation $\mathbf{f}$ for the induced map.

\subsection{Rescaling limits for an analytic family}
A non-Archimedean field is locally compact if and only if it is discretely valued and has finite residue field \cite{Ca}. Then $K$ is not locally compact, hence neither is $\mathbb{P}^1(K)$. Thus, we define the rescaling limits for an analytic family of rational maps over $K$ in the following sense:
\begin{definition}\label{rescaling}
Let $\{f_t\}$ be an analytic family of rational maps of nontrivial degree at least $2$. A moving frame $\{M_t\}$ is called a \textit{rescaling} for $\{f_t\}$ if there exist an integer $q\ge 1$, a  rational map $g:\mathbb{P}^1(K)\to\mathbb{P}^1(K)$ of nontrivial degree $d'\ge 2$ and a finite subset $S$ of $\mathbb{P}^1(K)$ such that, as $t\to 0$, 
$$M_t^{-1}\circ f_t^q\circ M_t(z)\to g(z)$$
pointwise on $\mathbb{P}^1(K)\setminus S$. We say $g$ is a \textit{rescaling limit} for $\{f_t\}$ in $\mathbb{P}^1(K)\setminus S$. The minimal $q\ge 1$ such that the above holds is called the \textit{period} of the rescaling $\{M_t\}$.
\end{definition}
Following Kiwi \cite{Ki}, we define the following equivalence relations on the set of all rescalings.
\begin{definition}
Two moving frames $\{M_t\}$ and $\{L_t\}$ are \textit{equivalent} if there exists $M\in\mathrm{Rat}_1(K)$ such that $M_t^{-1}\circ L_t\to M$ as $t\to 0$.
\end{definition}
\begin{definition}
Two rescalings $\{M_t\}$ and $\{L_t\}$ for an analytic family $\{f_t\}$ are \textit{dynamically dependent} if there exist an integer $l\ge 0$ and a nonconstant rational map $g$ such that $L_t^{-1}\circ f^l\circ M_t\to g$, as $t\to 0$, pointwise outside some finite set.
\end{definition}
If $\{M_t\}$ and $\{L_t\}$ are two equivalent rescalings for an analytic family $\{f_t\}$ , then they are dynamically dependent. The converse is not true in general.

\section{Reductions}\label{re}
Recall $K$ is any arbitrary complete algebraically closed non-Archimedean field. Let $g\in\mathcal{O}_K(z)$ be a rational map. Then reducing the coefficients of $g$ modulo $\mathcal{M}_K$ and canceling common factors, we get a rational map $\tilde g$ over the residue field $k$, which is called the reduction of $g$. Now we can define a map
$$\rho_K:\mathrm{Rat}_d(K)\to\mathrm{Rat}_{\le d}(k),$$
where $\mathrm{Rat}_{\le d}(k)$ is the space of degree at most $d$ rational maps over $k$, sending $g$ to its reduction $\tilde g$. We call $\rho_K$ the \textit{reduction map} for rational maps over $K$.\par
We first state an easy proposition and omit the proof.
\begin{proposition}\label{reduction}
Let $\phi(z),\psi(z)\in K(z)$ be rational maps, and let $\rho(\phi)$ and $\rho(\psi)$ be their reductions, respectively. Then\par
\begin{enumerate}
\item $\rho(\phi\cdot\psi)=\rho(\phi)\cdot\rho(\psi)$,
\item $\rho(\phi+\psi)=\rho(\phi)+\rho(\psi)$,
\item If $\deg\rho(\psi)\ge 1$, then $\rho(\phi\circ\psi)=\rho(\phi)\circ\rho(\psi)$.
\end{enumerate}
\end{proposition}
In Proposition \ref{reduction} $(3)$, if $\deg\rho(\psi)=0$, the situation is complicated. For example, let $K$ be the completion of the formal Puiseux series over $\C$ and define rational maps $\phi(z)=z^2/t^2$ and $\psi(z)=tz^2$ over $K$. Then $\rho(\phi\circ\psi)(z)=z^4$ but $(\rho(\phi)\circ\rho(\psi))(z)=\infty$ since $\rho(\phi)=\infty$.\par 
Recall that $\mathbb{L}$ is the field of Hahn series over $K$. Since $\mathbb{L}$ is an algebraically closed and complete non-Archimedean field, we can consider the reduction map $\rho_\mathbb{L}:\mathrm{Rat}_d(\mathbb{L})\to\mathrm{Rat}_{\le d}(K)$ of rational maps over $\mathbb{L}$. 
\begin{definition}
Let $\{f_t\}$ be an analytic family of degree $d\ge 1$ rational maps. We say $\{f_t\}$ has good reduction if the associated rational map $\mathbf{f}$ has \textit{good reduction}, that is, $\deg\rho_\mathbb{L}(\mathbf{f})=d$. Otherwise, we say $\{f_t\}$ has \textit{bad reduction}. If there is a moving frame $\{M_t\}\subset\mathbb{P}^3(K)$ such that $\{M_t^{-1}\circ f_t\circ M_t\}$ has good reduction, we say that $\{f_t\}$ has \textit{potentially good reduction}.
\end{definition}
Given $f=[f_1:f_2]\in\mathbb{P}^{2d+1}(K)$, we can write 
$$f=[f_1:f_2]=[H_f\hat f_1:H_f\hat f_2]=H_f[\hat f_1:\hat f_2]=H_f\hat f,$$
where $H_f=\gcd(f_1,f_2)$ is a homogeneous polynomial and $\hat f=[\hat f_1:\hat f_2]$ is a rational map of degree at most $d$.
\begin{proposition}\label{coeff-ptwise}
Suppose $\{f_t\}$ is an analytic family of degree $d\ge 1$ rational maps such that $f_t\to H_f\hat f$, as $t\to 0$, in $\mathbb{P}^{2d+1}(K)$. Then, as $t\to 0$, $f_t$ converges to $\hat f$ pointwise on $\mathbb{P}^1(K)\setminus\{H_f=0\}$.
\end{proposition}
\begin{proof}
Write $f_t=[P_t:Q_t]$ and $\hat f=[P:Q]$. Since $f_t$ converges to $H_f\hat f$ in $\mathbb{P}^{2d+1}(K)$, there is a $\lambda\in K\setminus\{0\}$ such that for any $[X:Y]\in\mathbb{P}^1(K)$, as $t\to 0$, $P_t(X,Y)$ converges to $\lambda H(X,Y)P(X,Y)$ and  $Q_t(X,Y)$ converges to $\lambda H(X,Y)Q(X,Y)$. So if $[X:Y]\not\in\{H_f=0\}$, we have 
$f_t([X:Y])$ converges to $[P(X,Y):Q(X,Y)]$. Hence $f_t$ converges to $\hat f$ pointwise on $\mathbb{P}^1(K)\setminus\{H_f=0\}$.
\end{proof}
\begin{corollary}
Let $\{f_t\}$ be an analytic family of degree $d\ge 2$ rational maps. If $\deg_0\rho_{\mathbb{L}}(\mathbf f)\ge 2$, then $\{M_t=z\}$ is a rescaling for $\{f_t\}$ with corresponding rescaling limit $\rho_{\mathbb{L}}(\mathbf{f})$.
\end{corollary}
\begin{proof}
Note as $t\to 0$ there is a homogeneous polynomial $H\in K[X,Y]$ such that $f_t$ converges to $H\rho_{\mathbb{L}}(\mathbf{f})$ in $\mathbb{P}^{2d+1}(K)$. The conclusion then follows Proposition \ref{coeff-ptwise}.
\end{proof}
The converse of Proposition \ref{coeff-ptwise} is also true. 
\begin{proposition}\label{ptwise-coeff}
Let $\{f_t\}$ be an analytic family of degree $d\ge 1$ rational map and let $S\subset\mathbb{P}^1(K)$ be a finite subset. Suppose $f_t$ converges to $\hat f$ pointwise, as $t\to 0$, on $\mathbb{P}^1(K)\setminus S$. Then there exists a homogeneous polynomial $H$ of degree $d-\deg\hat f$ with zeros in $S$ such that $f_t\to H\hat f$, as $t\to 0$, in $\mathbb{P}^{2d+1}(K)$. 
\end{proposition}
\begin{proof}
Let $\mathbf{f}$ be the associated rational map of $\{f_t\}$. Then there exists homogeneous polynomial $H$ such that $f_t(z)$ converges to $H\rho_{\mathbb{L}}(\mathbf{f})$, as $t\to 0$, in $\mathbb{P}^{2d+1}(K)$.
Thus, by Proposition \ref{coeff-ptwise}, $\hat f=\rho_{\mathbb{L}}(\mathbf{f})$. It is easy to check $H$ satisfies the required conditions.
\end{proof}
\begin{corollary}
Suppose $K$ has positive characteristic $p>0$. Let $\{f_t(z)\}\subset K(z)$ be an analytic family of rational maps of nontrivial degree at least $2$. Let $\mathbf{f}$ be the associated rational map of $\{f_t\}$. If $\mathbf{f}$ is inseparable, then all the rescaling limits of $\{f_t\}$ are inseparable.
\end{corollary} 
\begin{proof} 
Let $g$ be a rescaling limit for $\{f_t\}$. Then by Definition \ref{rescaling} and Proposition \ref{ptwise-coeff}, there exist a rescaling $\{M_t\}$, an integer $q\ge 1$ and homogeneous polynomial $H$ such that $M_t^{-1}\circ f_t^q\circ M_t\to Hg$. Note the associated rational map $\mathbf{M}^{-1}\circ\mathbf{f}^q\circ\mathbf{M}$ of $\{M_t^{-1}\circ f_t^q\circ M_t\}$ is inseparable since $\mathbf{f}$ is inseparable. Considering the coefficients of $g$, we have the map $g$ is inseparable.
\end{proof}

\section{Berkovich Dynamics}\label{bd}
In this section, we first summarize the properties of the dynamics on a Berkovich space, see \cite{BR,Ben,Jo,Ri}. Then we restate the results in \cite{Ki}, which are proven for a holomorphic family of rational maps over $\C$. These results are still true for an analytic family $\{f_t(z)\}$ of rational maps over a field $K$ with characteristic zero.\par
Recall that the Berkovich Julia set $J_{\mathrm{Ber}}(\phi)$ of a rational map $\phi:\mathbb{P}_{\mathrm{Ber}}^1(\mathbb{L})\to\mathbb{P}_{\mathrm{Ber}}^1(\mathbb{L})$ is the set consisting of all points $\xi\in\mathbb{P}_{\mathrm{Ber}}^1(\mathbb{L})$ such that $\cup_{n\ge 0}\phi^n(U)$ omits finitely many points of $\mathbb{P}_{\mathrm{Ber}}^1(\mathbb{L})$ for any neighborhood $U$ of $\xi$. The classical Julia set $J_{I}(\phi)$ is $J_{\mathrm{Ber}}(\phi)\cap\mathbb{P}^1(\mathbb{L})$. Let $\xi\in\mathbb{P}_{\mathrm{Ber}}^1(\mathbb{L})\setminus\mathbb{P}^1(\mathbb{L})$ be a periodic point of $\phi$ of period $n\ge 1$. The multiplier $\lambda$ of $\xi$ is defined by the local degree of $\phi^n$ at $\xi$, that is, $\lambda:=\deg_\xi(\phi^n)$. If $\lambda\ge 2$, we say $\xi$ is repelling. If a periodic point $\xi\in\mathbb{P}_{\mathrm{Ber}}^1(\mathbb{L})\setminus\mathbb{P}^1(\mathbb{L})$ is repelling, then $\xi$ is a type II point. Let $\sO\subset\mathbb{P}_{\mathrm{Ber}}^1(\mathbb{L})$ be a $n$-cycle of $\phi$. The basin of $\sO$ is the interior of the set of points $\xi\in\mathbb{P}_{\mathrm{Ber}}^1(\mathbb{L})$ such that, for all neighborhoods $U$ of $\sO$, the orbit of $\xi$ is eventually contained in $U$. \par
Recall that the tangent space $T_\xi\mathbb{P}_{\mathrm{Ber}}^1(\mathbb{L})$ is the set of all directions at $\xi\in\mathbb{P}_{\mathrm{Ber}}^1(\mathbb{L})$. Let $\phi:\mathbb{P}_{\mathrm{Ber}}^1(\mathbb{L})\to\mathbb{P}_{\mathrm{Ber}}^1(\mathbb{L})$ be a rational map. Then for any $\vec{v}\in T_\xi\mathbb{P}_{\mathrm{Ber}}^1(\mathbb{L})$, there is a unique $\vec{w}\in T_{\phi(\xi)}\mathbb{P}_{\mathrm{Ber}}^1(\mathbb{L})$ such that for any $\xi'$ sufficiently near $\xi$, $\phi(\xi')\in\mathbf{B}_{\phi(x)}(\vec{w})^-$. Thus the rational map $\phi$ induces a map 
$$\phi_\ast:T_{\xi}\mathbb{P}_{\mathrm{Ber}}^1(\mathbb{L})\to T_{\phi(\xi)}\mathbb{P}_{\mathrm{Ber}}^1(\mathbb{L})$$ 
sending the direction $\vec{v}$ to the corresponding direction $\vec{w}$.
\begin{proposition}\label{fundamental}\cite[Corollary 9.25, Theorem 9.26, Corollary 9.27, Proposition 9.41]{BR} 
Let $\phi:\mathbb{P}_{\mathrm{Ber}}^1(\mathbb{L})\to\mathbb{P}_{\mathrm{Ber}}^1(\mathbb{L})$ be a rational map of degree at least $1$. Then $\phi(\xi_G)=\xi_G$ if and only if $\deg\rho_{\mathbb{L}}(\phi)\ge 1$. Moreover,\par
\begin{enumerate}
\item Assume $\phi(\xi_G)=\xi_G$. Identifying $T_{\xi_G}\mathbb{P}_{\mathrm{Ber}}^1(\mathbb{L})$ to $\mathbb{P}^1(K)$, the following hold:
\begin{enumerate}
\item $\deg_{\xi_G}\phi=\deg\rho_{\mathbb{L}}(\phi)$,
\item at the Gauss point $\xi_G$, $\phi_\ast=\rho_{\mathbb{L}}(\phi)$ on $\mathbb{P}^1(K)$.
\end{enumerate}
\item For $\xi\in\mathbb{P}_{\mathrm{Ber}}^1(\mathbb{L})$ and $\vec{v}\in T_\xi\mathbb{P}_{\mathrm{Ber}}^1(\mathbb{L})$, the image $\phi(\mathbf{B}_\xi(\vec{v})^-)$ always contains $\mathbf{B}_{\phi(\xi)}(\phi_\ast\vec{v})^-$, and either $\phi(\mathbf{B}_\xi(\vec{v})^-)=\mathbf{B}_{\phi(\xi)}(\phi_\ast\vec{v})^-$ or $\phi(\mathbf{B}_\xi(\vec{v})^-)=\mathbb{P}_{\mathrm{Ber}}^1(\mathbb{L})$. There exists an integer $m\ge 1$ such that
\begin{enumerate}
\item if $\phi(\mathbf{B}_\xi(\vec{v})^-)=\mathbf{B}_{\phi(\xi)}(\phi_\ast(\vec{v}))^-$, then each $\zeta\in\mathbf{B}_{\phi(\xi)}(\phi_\ast(\vec{v}))^-$ has $m$ preimages in $\mathbf{B}_\xi(\vec{v})^-$, counting multiplicities;\par
\item if $\phi(\mathbf{B}_\xi(\vec{v})^-)=\mathbb{P}_{\mathrm{Ber}}^1(\mathbb{L})$, there is an integer $N\ge m$ such that each $\zeta\in\mathbf{B}_{\phi(\xi)}(\phi_\ast(\vec{v}))^-$ has $N$ preimages in $\mathbf{B}_\xi(\vec{v})^-$ and each $\zeta\in\mathbb{P}_{\mathrm{Ber}}^1(\mathbb{L})\setminus\mathbf{B}_{\phi(\xi)}(\phi_\ast(\vec{v}))^-$ has $N-m$ preimages in $\textbf{B}_\xi(\vec{v})^-$, counting multiplicities.
\end{enumerate}
\end{enumerate}
\end{proposition}
Based on Proposition \ref{coeff-ptwise} and Proposition \ref{fundamental}, we have 
\begin{proposition}\label{main-prop}\cite[Proposition 3.4, Lemma 3.6, Lemma 3.7]{Ki} 
Let $\{f_t(z)\}\subset K(z)$ be an analytic family of rational maps of nontrivial degree at least $2$, and let $\{M_t\}$ and $\{L_t\}$ be moving frames. Denote by $\mathbf{f}$, $\mathbf{M}$ and $\mathbf{L}$ the associated rational maps. Then\par
\begin{enumerate}
\item  For all $l\ge 1$, the following are equivalent:
\begin{enumerate}
\item There exists a rational map $g:\mathbb{P}^1(K)\to\mathbb{P}^1(K)$ of degree at least $d\ge 1$ such that $M_t^{-1}\circ f_t^l\circ M_t$ converges to $g$ pointwise, as $t\to 0$, on $\mathbb{P}^1(K)$ off a finite subset.
\item $\mathbf{f}^l(\xi)=\xi$, where $\xi=\mathbf{M}(\xi_G)$ and $\text{deg}_\xi\mathbf{f}^l=d$.
\end{enumerate}
In the case in which $(a)$ and $(b)$ hold, the map $(\mathbf{f}^l)_\ast:T_\xi\mathbb{P}_{\mathrm{Ber}}^1(\mathbb{L})\to T_\xi\mathbb{P}_{\mathrm{Ber}}^1(\mathbb{L})$ is conjugate via a $M\in\mathrm{Rat}_1(K)$ to $g:\mathbb{P}^1(K)\to\mathbb{P}^1(K)$.
\item Moving frames $\{M_t\}$ and $\{L_t\}$ are equivalent if and only if $\mathbf{M}(\xi_G)=\mathbf{L}(\xi_G)$.
\item The following are equivalent:
\begin{enumerate}
\item $\mathbf{f}\circ\mathbf{M}(\xi_G)=\mathbf{L}(\xi_G)$.
\item As $t\to 0$, $L_t^{-1}\circ f_t\circ M_t$ converges to some nonconstant rational map $g:\mathbb{P}^1(K)\to\mathbb{P}^1(K)$ pointwise outside some finite subset. 
\end{enumerate}
\end{enumerate}
\end{proposition}
\begin{corollary}
Let $\{f_t\}\subset K(z)$ be an analytic family of degree at least $2$ rational maps. Suppose $\{f_t\}$ has (potentially) good reduction. Then there is at most one rescaling, up to equivalence, for $\{f_t\}$, and this rescaling is of period $1$.
\end{corollary}
\begin{proof}
Let $\mathbf{f}$ be the associated rational map of $\{f_t\}$. Then $\mathbf{f}$ has (potentially) good reduction. Then the classical Julia set $J_I(\mathbf{f})=\emptyset$ and the Berkovich Julia set $J_{\mathrm{Ber}}(\mathbf{f})$ is a singleton set \cite[Lemma 10.53]{BR}. Thus $\phi$ has no repelling periodic points of type I and has only one repelling periodic point \cite[Theorems 10.81,10.82]{BR}. By Proposition \ref{main-prop}, all the rescalings of $\{f_t\}$ are equivalent and they are of period $1$.
\end{proof} 
To relate the critical points of $\mathbf{f}$ and the rescaling limits of $\{f_t\}$, we first state the following non-Archimedean Rolle's theorem:
\begin{lemma}\label{rolle}\cite[Application 1]{Fa2}
Suppose $K$ has characteristic zero and residue characteristic zero. Let $\phi\in K(z)$ be a rational map of degree at least $1$. If $\phi$ has two distinct zeros in the closed disk $\overline{D}(a,r)$, then it has a critical point in $\overline{D}(a,r)$.
\end{lemma}
We should mention here Lemma \ref{rolle} is not true in general. If $K$ has characteristic zero and residue characteristic $p>0$, then under same assumptions, $\phi$ is only guaranteed to have a critical point in $\overline{D}(a,r|p|_K^{-1/(p-1)})$ which is strictly larger than $\overline{D}(a,r)$. If $K$ has characteristic $p>0$, consider the field $\mathbb{L}$ and $\phi(z)=z^p-z\in\mathbb{L}(z)$. Then $\phi$ has only one critical point, which is $\infty\in\mathbb{P}^1(\mathbb{L})$. However,  $\phi$ has $p$ zeros in $\overline{D}(0,1)\subset\mathbb{L}$. For more details about rational maps with one critical point, we refer  \cite{Fa0}.\par
Applying the non-Archimedean Rolle's theorem and using the same proof in \cite{Ki}, we have
\begin{proposition}\label{inj-crit-0}
Suppose $K$ has characteristic zero. Consider a rational map $\phi:\mathbb{P}_{\mathrm{Ber}}^1(\mathbb{L})\to\mathbb{P}_{\mathrm{Ber}}^1(\mathbb{L})$ of degree at least $2$. Let $\xi\in\mathbb{P}_{\mathrm{Ber}}^1(\mathbb{L})$ be a type II point and let $\vec{v}\in T_\xi\mathbb{P}_{\mathrm{Ber}}^1(\mathbb{L})$. If $\phi$ is not injective on $\mathbf{B}_\xi(\vec{v})^-$, then there is a critical point of $\phi$ in $\mathbf{B}_\xi(\vec{v})^-$ such that the corresponding critical value $\phi(c)\in\mathbf{B}_{\phi(\xi)}(\phi_\ast(\vec{v}))^-$.
\end{proposition}
\begin{proposition}\label{crit-fin-0}
Suppose $K$ has characteristic zero. Consider a rational map $\phi:\mathbb{P}_{\mathrm{Ber}}^1(\mathbb{L})\to\mathbb{P}_{\mathrm{Ber}}^1(\mathbb{L})$ of degree at least $2$. Let $\sO$ be a type II periodic orbit of period $q\ge 1$ of $\phi$. Assume the basin of $\sO$ is free of critical points of $\phi$. Then, for all $\xi\in\sO$, every $\vec{v}\in T_\xi\mathbb{P}_{\mathrm{Ber}}^1(\mathbb{L})$ with $\phi^q(\mathbf{B}_\xi(\vec{v})^-)=\mathbb{P}_{\mathrm{Ber}}^1(\mathbb{L})$  has a finite forward orbit under $(\phi^q)_\ast$. Moreover, if $\deg (\phi^q)_\ast\ge 2$, then $(\phi^q)_\ast$ is postcritically finite.
\end{proposition}
In the next section, we establish analogs of these two propositions in positive characteristic, and from this deduce our main result.

\section{Rational maps over fields with positive characteristic}\label{pc}
Assume that the field $K$ has positive characteristic $p>0$. A nonconstant rational map $\phi\in K(z)$ can then be written as $\phi(z)=\phi_1(z^{p^j})$ for some integer $j\ge 0$, where $\phi_1$ is separable. Recall the ramification locus $\sR_\phi=\{\xi\in\mathbb{P}_{\mathrm{Ber}}^1(K): \deg_\xi\phi\ge 2\}$ and $\phi$ is tame if $\sR_\phi$ contains finitely many points with valence at least $3$. We say a rational map $\phi\in K(z)$ is tamely ramified if the characteristic of $K$ does not divide the multiplicity $m_\phi(z)$ for any $z\in\mathbb{P}^1(K)$. The space $\mathbb{P}_{\mathrm{Ber}}^1(K)\setminus\mathbb{P}^1(K)$ carries a natural metrizable topology, the strong topology, see \cite{BR, Fa1}. With respect to this metric, there exists $r>0$ such that the ramification locus $\sR_\phi$ is in an $r$-neighborhood of the connected hull $\mathrm{Hull}(\mathrm{Crit}(\phi))$ of critical set if and only if $\phi$ is tamely ramified \cite[Theorem E]{Fa2}. If $\phi$ is separable, the extreme case $\sR_\phi\subseteq\mathrm{Hull}(\mathrm{Crit}(\phi))$ is equivalent to $\phi$ is tame \cite[Corollary 7.13]{Fa1}.\par
We can prove the following non-Archimedean Rolle's theorem for a separable tame rational map over a field $K$ with positive characteristic.
\begin{lemma}\label{rolle-positive}
Suppose $K$ has positive characteristic $p>0$. Let $\phi\in K(z)$ be a separable tame rational map of degree at least $1$. If $\phi$ has two distinct zeros in the closed disk $\overline{D}(a,r)$, then it has a critical point in $\overline{D}(a,r)$.
\end{lemma}
\begin{proof}
Suppose there is no critical point in $\overline{D}(a,r)$. Let $\xi_{a,r}\in\mathbb{P}_{\mathrm{Ber}}^1(K)$ be the point corresponding to the closed disk $\overline{D}(a,r)$. Then $\xi_{a,r}\not\in\mathrm{Hull}(\mathrm{Crit}(\phi))$. Let $\vec{v}\in T_{\xi_{a,r}}\mathbb{P}_{\mathrm{Ber}}^1(K)$ be the direction such that $\overline{D}(a,r)\subset\mathbb{P}_{\mathrm{Ber}}^1(K)\setminus\mathbf{B}_{\xi_{a,r}}(\vec{v})^-$. Then the set $\mathbb{P}_{\mathrm{Ber}}^1(K)\setminus\mathbf{B}_{\xi_{a,r}}(\vec{v})^-$ is disjoint with $\mathrm{Hull}(\mathrm{Crit}(\phi))$. So 
$$\mathbb{P}_{\mathrm{Ber}}^1(K)\setminus\mathbf{B}_{\xi_{a,r}}(\vec{v})^-\cap\sR_\phi=\emptyset.$$ 
Since $\sR_\phi$ is closed, there exist $\xi\in\mathbb{P}_{\mathrm{Ber}}^1(K)$ and $\vec{w}\in T_\xi\mathbb{P}_{Ber}^1(K)$ such that $\mathbb{P}_{\mathrm{Ber}}^1(K)\setminus\mathbf{B}_{\xi_{a,r}}(\vec{v})^-\subset\mathbf{B}_{\xi}(\vec{w})^-$ and $\mathbf{B}_{\xi}(\vec{w})^-\cap\sR_\phi=\emptyset$. Hence $\phi$ is injective on $\mathbf{B}_{\xi}(\vec{w})^-$ \cite[Corollary 3.8]{Fa1}. So $\phi$ is injective on $\mathbb{P}_{\mathrm{Ber}}^1(K)\setminus\mathbf{B}_{\xi_{a,r}}(\vec{v})^-$. Thus $\phi$ is injective on the closed disk $\overline{D}(a,r)$. So $\phi$ has at most one zero in $\overline{D}(a,r)$. It is a contradiction.
\end{proof}
Applying Lemma \ref{rolle-positive} and the argument in \cite[Lemma 4.2]{Ki}, we obtain an analog of Proposition \ref{inj-crit-0}:
\begin{proposition}\label{inj-crit}
Suppose $K$ has positive characteristic $p>0$ and consider a separable tame rational map $\phi:\mathbb{P}_{\mathrm{Ber}}^1(\mathbb{L})\to\mathbb{P}_{\mathrm{Ber}}^1(\mathbb{L})$ of degree at least $2$. Let $\xi\in\mathbb{P}_{\mathrm{Ber}}^1(\mathbb{L})$ be a type II point and let $\vec{v}\in T_{\xi}\mathbb{P}_{\mathrm{Ber}}^1(\mathbb{L})$. If $\phi$ is not injective on $\mathbf{B}_{\xi}(\vec{v})^-$, then there is a critical point of $\phi$ in $\mathbf{B}_{\xi}(\vec{v})^-$ such that the corresponding critical value $\phi(c)\in\mathbf{B}_{\phi(\xi)}(\phi_\ast(\vec{v}))^-$.
\end{proposition}
We now prove an analogy of Proposition \ref{crit-fin-0}:
\begin{proposition}\label{crit-fin}
Suppose $K$ has positive characteristic $p>0$. Consider a rational map $\phi:\mathbb{P}_{\mathrm{Ber}}^1(\mathbb{L})\to\mathbb{P}_{\mathrm{Ber}}^1(\mathbb{L})$ of nontrivial degree at least $2$ and suppose the separable part $\phi_1$ of $\phi$ is tame. Let $\sO$ be a type II periodic orbit of period $q\ge 1$ of $\phi$. Assume the basin of $\sO$ is free of critical values of $\phi_1$. Then, for all $\xi\in\sO$, every $\vec{v}\in T_{\xi}\mathbb{P}_{\mathrm{Ber}}^1(\mathbb{L})$ with $\phi^q(\mathbf{B}_{\xi}(\vec{v})^-)=\mathbb{P}_{\mathrm{Ber}}^1(\mathbb{L})$ has a finite forward orbit under $(\phi^q)_\ast$. 
\end{proposition}
\begin{proof}
Let $\vec{v}\in T_{\xi}\mathbb{P}_{\mathrm{Ber}}^1(\mathbb{L})$ such that $\phi^q(\mathbf{B}_{\xi}(\vec{v})^-)=\mathbb{P}_{\mathrm{Ber}}^1(\mathbb{L})$ and $\vec{v}$ has an infinite forward orbit under $(\phi^q)_\ast$. We will show there exists a critical value of $\phi_1$ in the basin of $\sO$. Let $q_0\ge 1$ be the smallest integer such that 
$$\phi_1(\mathbf{B}_{\phi_2\circ\phi^{q_0}(\xi)}((\phi_2\circ\phi^{q_0})_\ast(\vec{v}))^-)=\mathbb{P}_{\mathrm{Ber}}^1(\mathbb{L}).$$ 
Then by Proposition \ref{fundamental} and Proposition \ref{inj-crit}, there is a critical point $c\in\mathrm{Crit}(\phi_1)$ such that $\phi_1(c)\in\mathbf{B}_{\phi^{{q_0}+1}(\xi)}((\phi^{{q_0}+1})_\ast(\vec{v}))^-$. Now we show for each $n\ge {q_0+1}$, $\mathbf{B}_{\phi^{n}(\xi)}((\phi^{n})_\ast(\vec{v}))^-$ contains a point in the forward orbit of a critical value of $\phi_1$. By induction, suppose it holds for $n=k\ge q_0+1$. If $\phi(\mathbf{B}_{\phi^{k}(\xi)}((\phi^{k})_\ast(\vec{v}))^-)=\mathbf{B}_{\phi^{k+1}(\xi)}((\phi^{k+1})_\ast(\vec{v}))^-$, then $\mathbf{B}_{\phi^{k+1}(\xi)}((\phi^{k+1})_\ast(\vec{v}))^-$ contains a point in the forward orbit of a critical value of $\phi_1$. If $\phi(\mathbf{B}_{\phi^{k}(\xi)}((\phi^{k})_\ast(\vec{v}))^-)=\mathbb{P}_{\mathrm{Ber}}^1(\mathbb{L})$, then $\phi_1(\mathbf{B}_{\phi_2\circ\phi^{k}(\xi)}((\phi_2\circ\phi^{k})_\ast(\vec{v}))^-)=\mathbb{P}_{\mathrm{Ber}}^1(\mathbb{L})$. By Proposition \ref{fundamental} and Proposition \ref{inj-crit}, $\mathbf{B}_{\phi^{k+1}(\xi)}((\phi^{k+1})_\ast(\vec{v}))^-$ contains a critical value of $\phi_1$.\par
Thus, for $n$ large, $\mathbf{B}_{\phi^{nq}(\xi)}((\phi^{nq})_\ast(\vec{v}))^-$ contains a point in the forward orbit of a critical value of $\phi_1$. Note for $n$ sufficiently large, say $n\ge n_0$, 
$$\phi(\mathbf{B}_{\phi^{nq}(\xi)}((\phi^{nq})_\ast(\vec{v}))^-)\not=\mathbb{P}_{\mathrm{Ber}}^1(\mathbb{L}).$$ 
Suppose $\phi^l(\phi_1(c))\in\mathbf{B}_{\phi^{n_0q}(\xi)}((\phi^{n_0q})_\ast(\vec{v}))^-$ for some $c\in \mathrm{Crit}(\phi_1)$ and $l\ge 0$, then $\phi^{nq+l}(\phi_1(c))\to\xi$, as $n\to\infty$, in the weak topology. Thus $\phi_1(c)$ is in the basin of the periodic cycle $\sO$.
\end{proof}
\begin{corollary}\label{crit-fin-rl}
Under the same assumptions in Proposition \ref{crit-fin}, if $\deg_0(\phi^q)_\ast\ge 2$, then $(\phi^q)_\ast$ is postcritically finite at the nontrivial critical points.
\end{corollary}
\begin{proof}
Suppose $(\phi^q)_\ast$ is not postcritically finite at the nontrivial critical points. Let $\vec{v}\in T_{\xi}\mathbb{P}_{\mathrm{Ber}}^1(\mathbb{L})$ be a nontrivial critical point of $(\phi^q)_\ast$ with infinite forward orbit, then there exists  $j\ge 0$ such that $(\phi_2\circ\phi^j)^{-1}(\mathrm{Crit}(\phi_1))\cap\mathbf{B}_{\xi}(\vec{v})^-\not=\emptyset$.  If $\phi^n(\mathbf{B}_{\xi}(\vec{v})^-)\not=\mathbb{P}_{\mathrm{Ber}}^1(\mathbb{L})$ for all $n\ge 1$, then 
$$\phi_1(\mathrm{Crit}(\phi_1))\cap\mathbf{B}_{\phi^{j+1}(\xi)}((\phi^{j+1})_\ast(\vec{v}))^-\not=\emptyset.$$
Hence for all $n\ge 0$, 
$$\phi^n(\phi_1(\mathrm{Crit}(\phi_1)))\cap\mathbf{B}_{\phi^{n+j+1}(\xi)}((\phi^{n+j+1})_\ast(\vec{v}))^-\not=\emptyset.$$ 
So there exists $c\in\mathrm{Crit}(\phi_1)$ such that $\phi_1(c)$ in the basin of $\sO$. If there exists $n_0\ge 1$ such that $\phi^{n_0}(\mathbf{B}_{\xi}(\vec{v})^-)=\mathbb{P}_{\mathrm{Ber}}^1(\mathbb{L})$, then $\vec{v}$ has a finite forward orbit by Proposition \ref {crit-fin}.
\end{proof}
Based on Proposition \ref{crit-fin-0} and Corollary \ref{crit-fin-rl}, applying the argument in \cite{Ki}, we can prove Theorem \ref{main-thm}.
\begin{proof}[Proof of Theorem \ref{main-thm}]
Let $\{M_t^{(1)}\},\cdots,\{M_t^{(n)}\}$ be pairwise dynamically independent rescalings for $\{f_t\}$ of periods $q_1,\cdots,q_n$ such that the corresponding rescaling limits are not postcritically finite at nontrivial critical points. Let $\mathbf{f}$ be the associated rational map of $\{f_t\}$ and $\mathbf{M}^{(1)},\cdots,\mathbf{M}^{(n)}$ be the associated rational maps of $\{M_t^{(1)}\},\cdots,\{M_t^{(n)}\}$. Let $\xi_j=\mathbf{M}^{(j)}(\xi_G)\in\mathbb{P}_{\mathrm{Ber}}^1(\mathbb{L})$ for $j=1,\cdots,n$. Then by Proposition \ref{main-prop}, for all $j=1,\cdots,n$, $\mathbf{f}\ ^{q_j}_\ast:T_{\xi_j}\mathbb{P}_{\mathrm{Ber}}^1(\mathbb{L})\to T_{\xi_j}\mathbb{P}_{\mathrm{Ber}}^1(\mathbb{L})$ is not postcritically finite at the nontrivial critical points, and the points $\xi_1,\cdots,\xi_n$ are in pairwise distinct periodic orbits of $\mathbf{f}$. Note the separable part of $\mathbf{f}$ has at most $2\deg_0\mathbf{f}-2$ critical points, hence it has at most $2\deg_0\mathbf{f}-2$ critical values. Then by Proposition \ref{crit-fin-0} for the case $char\ K=0$ and Corollary \ref{crit-fin-rl} for the case $char\ K>0$, we have $n\le 2\deg_0\mathbf{f}-2$.
\end{proof}

\section{Examples}\label{ex}
In this section, we give some examples to illustrate rescaling limits in non-Archimedean dynamics. We refer \cite{Ki} for more examples. All examples in \cite{Ki} are holomorphic families over $\C$, which can be considered as analytic families over non-Archimedean fields.\par 
Now let $p>0$ be a prime number. Denote by $K$ the field $\overline{\F}_p[[s^\mathbb{Q}]]$ of Hahn series over $\overline{\F}_p$ with respect to its nontrivial non-Archimedean absolute value. Then $char\ K=p>0$.
\begin{example}Polynomials with quadratic separable parts.
\end{example}
Given sufficiently small $t\in K\setminus\{0\}$. Consider the map
$$f_t(z)=(z^p-(s+t))(z^p-s)\in K(z).$$
Then associated map $\mathbf{f}$ of $\{f_t\}$ has separable part $\mathbf{f_1}(z)=(z-(s+t))(z-s)$. Note that $\mathbf{f_1}$ is tame if and only if $p\ge 3$. In fact, if $p=2$, $\mathbf{f_1}$ has only one critical point at $\infty$.\par
Note $\{f_t\}$ is an analytic family that has good reduction. Then, up to equivalence, the moving frame $\{M_t(z)=z\}$ is the unique possible rescaling. The corresponding limit is $g(z)=(z^p-s)^2$. If $p=2$, then $g(z)=z^4+s^2$ has nontrivial degree $1$. If $p\ge 3$, $g(z)=(z^p-s)^2$  has separable part $g_1(z)=(z-s)^2$. Note $\mathrm{Crit}(g_1)=\{s,\infty\}$ and $g_1(s)$ has an infinite forward orbit under $g$. Thus $g$ is a rescaling limit that is not postcritically finite at nontrivial critical points.

\begin{example}Connected Julia sets.
\end{example}
Let $a\in K$ with $0<|a|_K<1$ and $b\in K$ with $|b|_K=|b-1|_K=1$. Define 
$$\phi_{a,b}(z)=\frac{az^6+1}{az^6+z(z-1)(z-b)}\in K(z).$$
Then the Berkovich Julia set $J_{\mathrm{Ber}}(\phi_{a,b})$ is connected but not contained in a line segment \cite{BBC}.\par 
For sufficiently small $t\in K\setminus\{0\}$, let $a=t^q$, where $q=3p^2+1$, and fix $b\in K$. Define $f_t(z)=\phi_{t^q,b}(z^p)$. Then $\{f_t\}$ is a degenerated analytic family defined in a neighborhood of $t=0$. Let $\mathbf{f}$ be the associated map of $\{f_t\}$. Then the separable part $\mathbf{f_1}$ of $\mathbf{f}$ is $\mathbf{f_1}(z)=(t^qz^6+1)/(t^qz^6+z(z-1)(z-b))$, which has a connected Berkovich Julia set $J_{\mathrm{Ber}}(\mathbf{f_1})$.\par 
First the moving frame $\{M_t(z)=z\}$ is a rescaling of period $1$ with rescaling limit 
$$g(z)=\frac{1}{z(z-1)(z-b)}\circ z^p.$$
Note $\mathbf{f_1}$ maps the segment $[\xi_{0,|t|^{q/2}},\xi_G]$ isometrically onto the segment $[\xi_{0,|t|^{-q/2}},\xi_G]$. And $\mathbf{f_1}$ maps the segment $[\xi_G,\xi_{0,|t|^{-q/6}}]$ bijectively to the segment $[\xi_G,\xi_{0,|t|^{q/2}}]$ and the segment $[\xi_{0,|t|^{-q/6}},\xi_{0,|t|^{-q/3}}]$ bijectively to the segment $[\xi_{0,|t|^{q/2}},\xi_G]$, stretching by a factor of $3$, respectively. We may expect there exists a point $\xi_{0,r}\in\mathbb{P}^1_{\mathrm{Ber}}(\mathbb{L})$ such that $\mathbf{f}^2(\xi_{0,r})=\xi_{0,r}$. Indeed, we can choose $r=|t|_\mathbb{L}$. Let $L_t(z)=tz$. Then the moving frame $\{L_t(z)\}$ is a rescaling of period $2$ leading to rescaling limit 
$$h(z)=
\begin{cases}
\frac{1}{b^{6}z^3}\circ z^{4}, &\text{if}\ p=2,\\
-\frac{1}{b^{9}z}\circ z^{27}, &\text{if}\ p=3,\\
-\frac{1}{b^{3p}z^3}\circ z^{p^2}, &\text{if}\ p\ge 5.
\end{cases}$$ \par

\begin{example}McMullen maps.
\end{example}
This example is an analog of \cite[2.4]{Ki}. Given sufficiently small $t\in K\setminus\{0\}$, consider the map 
$$f_t(z)=z^{2p}+\frac{t^{1+2p^2}}{z^p}\in K(z).$$
Then $\{f_t\}$ is a degenerate analytic family defined in a neighborhood of $t=0$. The associated map $\mathbf{f}$ of $\{f_t\}$ has separable part $\mathbf{f_1}(z)=z^{2}+t^{1+2p^2}/z$. By \cite[Corollary 6.6]{Fa1}, when $p\ge 5$, the map $\mathbf{f_1}(z)$ is tame. In fact, when $p=3$, at every type II point, $\phi$ has separable reduction. Then by \cite[Corollary 7.13]{Fa1}, $\phi$ is tame. When $p=2$, $\phi$ has a unique critical point. Thus, $\phi$ is tame if and only if $p\ge 3$.\par
Obviously, the moving frame $\{M_t(z)=z\}$ is a rescaling of period $1$. The corresponding rescaling limit is $g(z)=z^{2p}$, which has a tame separable part $g_1(z)=z^2$. Note the nontrivial critical set $\mathrm{Crit}_0(g)=\{0,\infty\}$. So the rescaling limit $g$ is postcritically finite at the nontrivial critical points.\par
Moreover, the moving frame $\{L_t(z)=tz\}$ is a rescaling of period $2$ for $\{f_t\}$, which leads to the rescaling limit $h(z)=z^{-2p^2}$. The rescaling limit $h(z)$ is also postcritically finite at the nontrivial critical points.\par

\subsection*{Acknowledgements}
The author is grateful to the anonymous referee for valuable comments. The author would like to thank Jan Kiwi for explanation of his work and useful suggestions. The author would also like to thank Robert Benedetto, Laura DeMarco and Xander Faber for useful comments.

\bibliographystyle{siam}
\bibliography{references}

\end{document}